\documentclass[12pt]{amsart}
\usepackage{graphicx}
\usepackage{geometry}
\usepackage{amssymb, amsmath, amsfonts}
\usepackage{url}
\usepackage[pdfborder={0 0 0}]{hyperref}
\usepackage{verbatim} 
\usepackage[all]{xy}
\usepackage{color}

\usepackage{marginnote}

\allowdisplaybreaks
\newtheorem{theorem}{Theorem}[section]
\newtheorem{prop}[theorem]{Proposition}
\newtheorem{lemma}[theorem]{Lemma}
\newtheorem{definition}[theorem]{Definition}

\numberwithin{equation}{section}

\newif\ifcolor
\colortrue  

\begin{document}
	
	\raggedbottom
	
	\title{Regularity of solutions of the parabolic normalized $p$-Laplace equation}
	
	\author[F. H\o eg]{Fredrik Arbo H\o eg}
	\address{Department of Mathematical Sciences\\
		Norwegian University of Science and Technology\\
		NO-7491 Trondheim\\ Norway}
	\email{fredrik.hoeg@ntnu.no}
	
	\author[P. Lindqvist]{Peter Lindqvist}
	\address{Department of Mathematical Sciences\\ Norwegian University of Science and Technology\\ NO-7491 Trondheim\\ Norway}
	\email{peter.lindqvist@ntnu.no}

	\maketitle
	
	
	\begin{abstract}
		The parabolic normalized p-Laplace equation is studied. We prove that a viscosity solution has a time derivative in the sense of Sobolev belonging locally to $L^2$. 
	\end{abstract}

	\section{Introduction}
	We consider viscosity solutions of the \textit{normalized p-Laplace} equation  \\
	\begin{equation}
	\frac{\partial u}{\partial t}\, = \,|\nabla u|^{2-p} \text{div}\left( |\nabla u|^{p-2} \nabla u \right), \quad 1< p < \infty,\\[1em]
	\label{pLaplace}
	\end{equation}
	in  $\Omega_T= \Omega \times (0,T)$,\, $\Omega$ being a domain in $\mathbb{R}^n$. Formally, the equation reads
        \begin{equation*}
          \frac{\partial u}{\partial t}\, =\,\Delta u + (p-2)|\nabla u|^{-2}\sum_{i,j=1}^n	\frac{\partial u}{\partial x_i}	\frac{\partial u}{\partial x_{j}}\frac{\partial^2 u}{\partial x_{i} \partial x_j}.
          \end{equation*}
          In the  linear case $p=2$ we have the Heat Equation $ u_t = \Delta u$ and also  for $n=1$ the equation reduces to the Heat Equation $u_t = (p-1)u_{xx}.$ At the limit $p=1$ we obtain the equation for motion by mean curvature. We aim at showing that the time derivative $\frac{\partial u}{\partial t}$   exists in the Sobolev sense and belongs to $L^2_{\text{loc}}(\Omega_T)  $. We also study the second derivatives $\frac{\partial^2 u}{\partial x_i \partial x_j}$.

          There has been some recent interest in connexion with Stochastic Game Theory, where the equation appears, cf. [MPR]. From our point of view the work [D] is of actual interest, because there it is shown that the time derivative $u_t$ of the viscosity solutions exists and is locally bounded, provided that the lateral boundary values are smooth. Thus the boundary values control the time regularity. If no such assumptions about the behaviour at the lateral boundary $\partial \Omega \times (0,T)$ are made, a conclusion like $u_t \in L^{\infty}_{\text{loc}}(\Omega_T)$ is in doubt. Our main result is the following, where we unfortunately have to restrict $p$:

	\begin{theorem}
		Suppose that $u=u(x,t)$ is a viscosity solution of the  \textit{normalized p-Laplace} equation in $\Omega_T.$
		If $\frac{6}{5}<p < \frac{14}{5}$, then the Sobolev derivatives $\frac{\partial u}{\partial t}$ and $\frac{\partial^2 u}{\partial x_{i} \partial x_j}$  exist   and belong to $ L^2_{\emph{loc}}(\Omega_T)$.
		\label{teor1}
	\end{theorem}

       We emphasize that no assumptions on boundary values are made for this interior estimate. Our method of proof is based on a verification of the identity
 	
        $$\int_0^T \!\!\int_{\Omega} u\phi_t\,dx dt\,=\, -\int_0^T \!\!\int_{\Omega} U\phi\,dx dt,\qquad \phi \in C^{\infty}_0(\Omega_T),$$
        where we have to prove that the function $U,$ which is the right-hand side of equation (\ref{pLaplace}), belongs to $L^2_{\text{loc}}(\Omega_T)$. Thus the second spatial derivatives $D^2u$ are crucial (local boundedness of $\nabla u$ was proven in [D], [BG] and interior H\"{o}lder estimates for the gradient in [JS]). The elliptic case has been studied in [APR].

	In the range $1<p<2$ one can bypass the question of second derivatives. 
    \begin{theorem}
    	Suppose that $u=u(x,t)$ is a viscosity solution of the  \textit{normalized p-Laplace} equation in $\Omega_T.$
    	If $1<p<2$, then the Sobolev derivative $\frac{\partial u}{\partial t}$  exists and belongs to $ L^2_{\emph{loc}}(\Omega_T)$.
    	\label{teor2}
    \end{theorem}

	To avoid the problem of vanishing gradient, we
	first study the regularized  equation 
	\begin{equation}
	\frac{\partial u^\epsilon}{\partial t} = (|\nabla u^\epsilon|^2 + \epsilon^2)^{\frac{2-p}{2}} \text{div}\left( (|\nabla u^\epsilon|^2 + \epsilon^2 )^{\frac{p-2}{2}} \nabla u^\epsilon \right). 
	\label{pLaplaceepsilon}
	\end{equation}
        Here the classical parabolic regularity theory is applicable. The equation was studied by K. Does in [D], where an estimate of the gradient  $\nabla u^\epsilon$ was found with Bernstein's method. We shall prove  a maximum principle for the gradient. 
        Further, we differentiate equation \eqref{pLaplaceepsilon} with respect to the space variables and derive estimates for $u^\epsilon$ which are passed over to the solution $u$ of \eqref{pLaplace}.

        Analogous results seem to be possible to reach through the \emph{Cordes condition}. It also  restricts the range of valid exponents $p$. We have refrained from this approach, mainly since the absence of zero (lateral) boundary values produces many undesired  terms to estimate. Finally, we mention that the limits $\frac{6}{5}$ and $\frac{14}{5}$ in Theorem \ref{teor1} are evidently an artifact of the method. It would be interesting to know whether the theorem is valid in the whole range $1<p<\infty.$ In any case, our method is not capable to reach all exponents.
        
\bigskip
        \textbf{Acknowledgements.} Supported by the Norwegian Research Council (grant 250070). We thank Amal Attouchi for valuable help with a proof. 

	\section{Preliminaries}
        
        \textbf{Notation.} The gradient of a function $f: \Omega_T \rightarrow \mathbb{R}$ is \\ $$\nabla f= \left( \frac{\partial f}{\partial x_1},...,\frac{\partial f}{\partial x_n} \right) $$ 
        and its Hessian matrix is  \\ $$\left( D^2 f\right)_{ij} =\frac{\partial^2 f}{\partial x_i \partial x_j}, \qquad |D^2 f|^2 =\sum_{i,j=1}^n \Big(\frac{\partial^2 f}{\partial x_i \partial x_j} \Big)^2. $$ \\ We shall, occasionally, use the abbreviation 
	\begin{align*}
	u_j= \frac{\partial u}{\partial x_j}, \quad u_{jk}=\frac{\partial^2 u}{\partial x_j \partial x_k}
	\end{align*}
	for partial derivatives.  Young's inequality
        $$|ab|\,\leq\,\delta\,\frac{|a|^p}{p}+ \Big(\frac{1}{\delta}\Big)^{q-1}\frac{|b|^q}{q},\qquad \frac{1}{p}+\frac{1}{q} = 1$$
        is often referred to. Finally, the  summation convention is used when convenient.\\

\textbf{Viscosity solutions.} The normalized $p$-Laplace Equation is not in divergence form. Thus the concept of weak solutions with test functions under the integral sign is problematic. Fortunately, the modern concept of viscosity solutions works well.	Existence and uniqueness of viscosity solutions of the normalized $p$-Laplace equation was established in [BG].
	We recall the definition. 
	
	\begin{definition}
	We say that an upper semi-continuous function  $u$ is a \emph{viscosity subsolution} of equation \eqref{pLaplace} if for all $\phi \in C^2(\Omega_T)$ we have \\
	\begin{align*}
	\phi_t \leq \left ( \delta_{ij} + (p-2)\frac{\phi_{x_i} \phi_{x_j}}{|\nabla \phi|^2}    \right) \phi_{x_i x_j} \\
	\end{align*}
	at any interior point $(x,t)$ where $u-\phi$ attains a local maximum, provided $\nabla \phi(x,t) \neq 0$. Further,  at any interior point $(x,t)$ where $u-\phi$ attains a local maximum and $\nabla \phi(x,t)=0$ we require \\
	\begin{align*}
	\phi_t \leq \left ( \delta_{ij} + (p-2) \eta_i \eta_j    \right) \phi_{x_i x_j} \\
	\end{align*}
	for some $\eta \in \mathbb{R}^n$ with $|\eta| \leq 1.$
	 \\
	\end{definition}
	\begin{definition}
	 We say that a lower semi-continuous function  $u$ is a \emph{viscosity supersolution} of equation \eqref{pLaplace} if for all $\phi \in C^2(\Omega_T)$ we have \\
	 \begin{align*}
	 \phi_t \geq \left ( \delta_{ij} + (p-2)\frac{\phi_{x_i} \phi_{x_j}}{|\nabla \phi|^2}    \right) \phi_{x_i x_j} \\
	 \end{align*}
	 at any interior point $(x,t)$ where $u-\phi$ attains a local minimum, provided $\nabla \phi(x,t) \neq 0$. Further,  at any interior point $(x,t)$ where $u-\phi$ attains a local minimum and $\nabla \phi(x,t)=0$ we require \\
	\begin{align*}
	\phi_t \geq \left ( \delta_{ij} + (p-2) \eta_i \eta_j    \right) \phi_{x_i x_j} \\
	\end{align*}
	for some $\eta \in \mathbb{R}^n$ with $|\eta| \leq 1$.
	
	 \end{definition}
	 \begin{definition}
	 	
	 A continuous function $u$ is a \emph{viscosity solution} if it is both a viscosity subsolution and a viscosity supersolution. 
	\label{defvisc}
	\end{definition}
	For a detailed discussion on the definition at critical points we refer to Evans and Spruck [ES]. The reason behind the choice of $\eta\in \mathbb{R}^n$ is given in [ES] section 2. Viscosity solutions of equation \eqref{pLaplaceepsilon} are defined in a similar manner, except that now $\nabla \phi(x,t)=0$ is not a problem.\\

	 \textbf{Maximum Principle for the Gradient.}  In order to estimate the time derivative we need bounds on the second derivatives of $u^\epsilon$ (and also on its gradient). If we first assume that $u^\epsilon$ is $C^1$ on the parabolic boundary $\partial_{\text{par}}\Omega_T$, we get bounds on the gradient in all of $\Omega_T$. This follows from the following maximum principle.
          
	 	\begin{prop}[Maximum Principle]
		  Let $u^\epsilon$ be a solution of equation \eqref{pLaplaceepsilon}. If $\nabla u^\epsilon \in C^1(\overline{\Omega}_T)$, then
                 \begin{align*}
		\max_{\overline{\Omega}_T}\left\{|\nabla u^\epsilon|\right\}\, =\, \max_{\partial_{\emph{par}}\Omega_T}\left\{|\nabla u^\epsilon|\right\}.
	\end{align*}
		\label{maximumforv}
	\end{prop}

\begin{proof}
With some modifications a proof can be extracted from [D]. We give a direct proof.  To this end,  consider $$V^\epsilon(x,t)=|\nabla u^\epsilon|^2 + \epsilon^2,$$ 
	To  find the partial differential equation satisfied by $V^\epsilon$, we calculate\footnote{Sum over repeated indices.}
	\begin{align*}
	V^\epsilon_i\,=\,2 u^\epsilon_\nu u^\epsilon_{i\nu}, \quad V^\epsilon_{ij} \,=\, 2 u^\epsilon_{\nu j} u^\epsilon_{i \nu} + 2 u^\epsilon_{\nu} u^\epsilon_{ij \nu} \\
	u^\epsilon_i u^\epsilon_j V^\epsilon_{ij}\, = \,\frac{1}{2}|\nabla V^\epsilon|^2 + 2u^\epsilon_i u^\epsilon_j u^\epsilon_\nu u^\epsilon_{ij \nu}. 
	\end{align*}
	Writing   equation \eqref{pLaplace} in the form $$u^\epsilon_t\,=\,\Big( \delta_{ij} + (p-2)\frac{u^\epsilon_i u^\epsilon_j}{|\nabla u^\epsilon|^2 + \epsilon^2}\Big) u^\epsilon_{ij},$$ we find \\
	\begin{align*}
	\frac{1}{2}V^\epsilon_t&\,=\, u^\epsilon_\nu \frac{\partial }{\partial x_\nu} u^\epsilon_t \,=\, u^\epsilon_\nu \Delta u^\epsilon_\nu - \frac{p-2}{2(V^\epsilon)^2} \left| \left \langle \nabla u^\epsilon, \nabla V^\epsilon \right \rangle \right |^2   \\[1em]
	& \quad{}+ \frac{p-2}{V^\epsilon}\Big( \frac{1}{4}|\nabla V^\epsilon|^2 + \frac{1}{2}u^\epsilon_\nu u^\epsilon_\mu V^\epsilon_{\nu \mu}   \Big).\\
	\end{align*}
	Rearranging and using
	$$\Delta V^\epsilon\,=\,2|D^2u^{\epsilon}|^2 + 2 \langle\nabla u^{\epsilon},\nabla \Delta u^{\epsilon}\rangle$$ we arrive at the following differential equation for $V^\epsilon$:\\
	\begin{equation}
	V^\epsilon_t\,=\,\Delta V^\epsilon - 2|D^2 u^\epsilon|^2 - \frac{p-2}{(V^\epsilon)^2}\left | \left \langle \nabla u^{\epsilon}, \nabla V^{\epsilon} \right \rangle \right |^2 + \frac{p-2}{V^\epsilon}\Big\{ \frac{1}{2}|\nabla V^\epsilon|^2 + u^\epsilon_\nu u^\epsilon_\mu V^\epsilon_{\nu \mu}   \Big\}. \\[1em] 
	\label{vmaximum}
	\end{equation}
	
	Let
	$$w(x,t)= |\nabla u^\epsilon(x,t)|^2 + \epsilon^2  - \alpha t = V^\epsilon(x,t) - \alpha t \quad \text{for}\quad \alpha>0.$$
	Suppose that $w^{\epsilon}$ has an \emph{interior} maximum point at $(x_0,t_0)$. At this point $V^\epsilon(x_0,t_0)>0$, otherwise we would have $V^\epsilon(x,t) \equiv 0$ in $\Omega_T$ in which case there is nothing to prove. By the infinitesimal calculus,
	$$\nabla w(x_0,t_0) = 0,\, \leq 0\quad\text{ and}\quad w_t(x_0,t_0) \geq 0,$$
	where we have included the case $t_0=T$. Further, the matrix $D^2w(x_0,t_0)$ is negative semidefinite.
	Using equation \eqref{vmaximum} and noting that $\nabla w= \nabla V^\epsilon$ and $D^2w = D^2 V^\epsilon,$ we get at $(x_0,t_0)$ \\
	\begin{align*}
	0  \,\leq\, w_t&= \, V^\epsilon_t- \alpha\, \\[1em]
	&= \,\Delta V^\epsilon - 2|D^2 u^\epsilon|^2 - \frac{p-2}{(V^\epsilon)^2}\left | \left \langle \nabla u^\epsilon, \nabla V^\epsilon \right \rangle \right |^2\\ &+ \frac{p-2}{V^\epsilon}\left \{ \frac{1}{2}|\nabla V^\epsilon|^2 + u^\epsilon_\nu u^\epsilon_\mu V^\epsilon_{\nu \mu}   \right\} - \alpha \\
	&\, = \,\Big( \delta_{ij} + (p-2) \frac{u^\epsilon_i u^\epsilon_j}{V^\epsilon} \Big) w^\epsilon_{ij} - 2|D^2 u^\epsilon|^2- \alpha \,\leq\, -\alpha  
	\end{align*}  \\
	since the matrix $A$ with elements $A_{ij} = \delta_{ij} + (p-2) \frac{u^\epsilon_i u^\epsilon_j}{V^\epsilon}$ is positive semidefinite. To avoid the contradiction $\alpha \leq 0$,  $w$ must attain its maximum on the parabolic boundary.\\
	
	Hence, for any $(x,t) \in \Omega_T$ we have  \\
	\begin{align*}
	V^\epsilon(x,t) - \alpha t \leq \max_{\partial_{\text{par}}\Omega_T} \left \{  V^\epsilon(x,t) - \alpha t \right \} \leq \max_{\partial_{\text{par}}\Omega_T}  V^\epsilon(x,t). \\
	\end{align*}
	We finish the proof by sending $\alpha \rightarrow 0^+$. 
\end{proof}

	With no assumptions for $u^\epsilon$ on the parabolic boundary, we need a stronger result taken from [D] p.381. 
	
	\begin{theorem}
	Let $u^\epsilon$ be a solution of equation \eqref{pLaplaceepsilon}, with $u^\epsilon(x,0)=u_0(x)$. Then
	\begin{align*}
	|\nabla u^\epsilon(x,t)| \leq C_{n,p} ||u_0||_{L_\infty(\Omega_T)}\left \{1+ \left(  \frac{1}{\emph{dist}((x,t),\partial_\emph{par} \Omega_T)}\right)^2  \right \}.
	\end{align*}
	\label{Kerstingjor}
	\end{theorem}
	Note that no condition on the lateral boundary $\partial \Omega \times [0,T]$ was used. By continuity,
	$$ |\nabla u^\epsilon(x,t)| \leq C_{n,p} ||u^\epsilon(\cdot, t_0)||_\infty\left \{1+ \left(  \frac{1}{\text{dist}((x,t),\partial_\text{par} \Omega_T)}\right)^2  \right \}$$ for $x \in D \subset \subset \Omega$ and $0<t_0\leq t \leq T-t_0.$
        The estimate
        \begin{equation}\label{gradbound}
        ||\nabla u ^\epsilon||_{L^\infty(D \times [t_0, T-t_0]  )} \leq C ||u^\epsilon||_{L^\infty(\Omega_T)}\left \{1+ \left(  \frac{1}{\text{dist}(D,\partial_\text{par} \Omega_T)}\right)^2  \right \}
        \end{equation}
        follows. (Here one can pass to the limit as  $\epsilon \rightarrow 0.$)
	
	The proof of the lemma below, a simple special case of the Miranda - Talenti lemma, can be found for smooth functions in [E] p. 308. If $f$ is not smooth, we perform a \emph{strictly} interior approximation, so that no boundary inegrals appear (which is possible since $\xi \in C_0^\infty$). 	
	\begin{lemma}[Miranda - Talenti]
	Let $\xi \in C_0^\infty(\Omega_T)$ and $f\in L^2(0,T,W^{2,2}(\Omega))$.  Then 
	\begin{align*}
\int_0^T\!\!\int_{\Omega} |\Delta (\xi f)|^2\, dx dt = \int_0^T\!\! \int_{\Omega} |D^2 (\xi f)|^2\, dx dt.
	\end{align*} 
	\label{evanstalenti}
	\end{lemma}

	\section{Regularization}
	
	The next lemma tells us that  solutions of \eqref{pLaplaceepsilon} converge locally uniformly to the viscosity solution of \eqref{pLaplace}.  
	
	\begin{lemma}
		Let $u$ be a viscosity solution of equation \eqref{pLaplace} and let $u^\epsilon$ be the classical solution of the regularized equation \eqref{pLaplaceepsilon} with boundary values
		\begin{align*}
		u=u^\epsilon  \quad \emph{on} \quad \partial_{par} \Omega_T.
		\end{align*}	 	
		Then 
		$u^\epsilon \rightarrow u$ uniformly on compact subsets of $\Omega_T$.
				\label{lemma1}
	\end{lemma}	
	\begin{proof}
	By Theorem \ref{Kerstingjor} we can use Ascoli's Theorem to extract a convergent subsequence  $u^{\epsilon_j}$ converging locally uniformly to some continuous function:  $u^{\epsilon_j} \rightarrow v$. We claim that $v$ is a viscosity solution of equation \eqref{pLaplace}. The lemma then follows by uniqueness.\\
		
		We demonstrate that $v$ is a viscosity subsolution. (A symmetric proof shows that $v$ is a viscosity supersolution.) Assume that $v-\phi$ attains a strict local maximum at $z_0=(x_0,t_0)$. Since $u^\epsilon \rightarrow v$ locally uniformly, there are points 
		\begin{align*}
		z_\epsilon \rightarrow z_0
		\end{align*}
		such that $u^\epsilon - \phi$ attains a local maximum at $z_\epsilon$. If $\nabla \phi (z_0) \neq 0$, then  $\nabla \phi (z_\epsilon) \neq 0$  for all $\epsilon>0$ small enough, and at $z_\epsilon$ we have 
		\begin{equation}
		\phi_t \leq \left(\delta_{ij} + (p-2) \frac{\phi_{x_i} \phi_{x_j}}{|\nabla \phi|^2 + \epsilon^2} \right) \phi_{x_i x_j}. 
		\label{viscsolnv}
		\end{equation}
		Letting $\epsilon \rightarrow 0,$ we see that $v$ satisfies Definition \ref{defvisc} when $\nabla \phi(z_0) \neq 0$. If $\nabla \phi(z_0)=0$, let 
		\begin{align*}
		\eta_\epsilon = \frac{\nabla \phi(z_\epsilon)}{\sqrt{  |\nabla \phi(z_\epsilon)|^2 + \epsilon^2}}. 
	\end{align*}
		Since $|\eta_\epsilon |\leq 1,$ there is a subsequence so that $\eta_{\epsilon_k} \rightarrow \eta$ when $k\rightarrow \infty$ for some $\eta \in \mathbb{R}^n$ with $|\eta| \leq 1$. Passing to the limit $\epsilon_k \rightarrow 0$ in equation \eqref{viscsolnv}, we see that $v$ is a viscosity subsolution. 
		\end{proof}

	Our proof of Theorem \ref{teor1} consists in  showing that the second derivatives $D^2 u^\epsilon$ belong locally to $L^2$ with a bound independent of $\epsilon.$ Once this is established, we see that 
	\begin{align*}
	&\left( |\nabla u^\epsilon|^2 + \epsilon^2\right)^{\frac{2-p}{2}}\text{div}\left( \left( |\nabla u^\epsilon|^2 + \epsilon^2\right)^{\frac{p-2}{2}} \nabla u^\epsilon      \right) \\
	&= \Delta u^\epsilon + \frac{p-2}{|\nabla u^\epsilon|^2+ \epsilon^2} \left \langle \nabla u^\epsilon, D^2 u^\epsilon \nabla u^\epsilon \right \rangle \leq C_{p,n}|D^2 u^\epsilon|.  \\
	\end{align*}
	Hence, for any bounded subdomain $ D \subset \subset  \Omega_T$
	\begin{align*}
	\left | \left | \left( |\nabla u^\epsilon|^2 + \epsilon^2\right)^{\frac{2-p}{2}}\text{div}\left( \left( |\nabla u^\epsilon|^2 + \epsilon^2\right)^{\frac{p-2}{2}} \nabla u^\epsilon      \right) \right | \right |_{L^2(D)}\, \leq \, C, \\
	\end{align*}
with $C$	independent of $\epsilon$. By this uniform bound, there exists a subsequence such that, as $j \rightarrow  \infty$, 
	\begin{align*}
	\left( |\nabla u^{\epsilon_j}|^2 + \epsilon_j^2\right)^{\frac{2-p}{2}}\text{div}\left( \left( |\nabla u^{\epsilon_j}|^2 + \epsilon_j^2\right)^{\frac{p-2}{2}} \nabla u^{\epsilon_j}      \right) \rightarrow U \qquad \text{weakly in $L^2(D)$.} \\
	\end{align*}
	In particular, this means that $U \in L^2(D)$ and for any $\phi \in C_0^\infty (D)$ we have 
	\begin{align*}
	\lim_{j \rightarrow \infty} \int_0^T\!\!\int_D \phi \left( |\nabla u^{\epsilon_j}|^2 + \epsilon_j^2\right)^{\frac{2-p}{2}}\text{div}\!\left( \left( |\nabla u^{\epsilon_j}|^2 + \epsilon_j^2\right)^{\frac{p-2}{2}} \nabla u^{\epsilon_j}      \right) dx dt\, = \,\int_0^T\!\!\int_D \phi\, U \, dx dt.\\
	\end{align*}
	
	If $u$ is the unique viscosity solution of \eqref{pLaplace}, we invoke Lemma \ref{lemma1} and the calculations above to find, for any test function $\phi \in C_0^\infty(D)$, 
	\begin{align*}
  \int_0^T\!\!\int_D  u \frac{\partial \phi}{\partial t} \, dx dt& = \lim_{j \rightarrow \infty}  \int_0^T\!\!\int_D   u^{\epsilon_j} \frac{\partial \phi}{\partial t} \, dx dt \\[1em]
	& = - \lim_{j \rightarrow \infty} \int_0^T\!\!\int_D   \phi \left( |\nabla u^{\epsilon_j}|^2 + \epsilon_j^2\right)^{\frac{2-p}{2}}\text{div}\left( \left( |\nabla u^{\epsilon_j}|^2 + \epsilon_j^2\right)^{\frac{p-2}{2}} \nabla u^{\epsilon_j}      \right) \, dx dt \\[1em]
	&= -\int_0^T\!\!\int_D \phi\, U \, dx dt.
	\end{align*}		
			This shows that the Sobolev derivative  $u_t$ exists and, since the previous equation holds for any subdomain $D \subset \subset \Omega_T,$ we conclude that $\frac{\partial u}{\partial t} = U \in L_\text{loc}^2(\Omega_T)$. --- To finish the proof of Theorem \ref{teor1}
	it remains to establish the missing  local bound of $\|D^2u^\epsilon\|_{L^2}$ uniformly in $\epsilon.$

	\section{The differentiated equation}\label{differentiated}
        We shall derive a fundamental identity. Let
        $$v^\epsilon= |\nabla u^\epsilon|^2,\quad V^\epsilon = |\nabla u^\epsilon|^2 + \epsilon^2.$$ 
		Differentiating equation \eqref{pLaplaceepsilon} with respect to the variable $x_j$ we obtain 
	\begin{align*}
	\frac{\partial }{\partial t}\, u^\epsilon_j &= \frac{2-p}{2}\left(V^\epsilon\right)^{-\frac{p}{2}} v^\epsilon_j \, \text{div}\left( \left( V^\epsilon\right)^{\frac{p-2}{2}} \nabla u^\epsilon \right) 
	& + \left(V^\epsilon\right)^{\frac{2-p}{2}}\text{div}\Big[ \left( \left( V^\epsilon\right)^{\frac{p-2}{2}} \nabla u^\epsilon \right)_j \Big].
	\end{align*}
		Take $\xi \in C_0^{\infty}(\Omega_T)$, with $\xi \geq 0$. Multiply both sides of the equation by $\xi^2 V^\epsilon u^\epsilon_j$ and sum $j$ from $1$ to $n$. Integrate over $\Omega_T$, using integration by parts and keeping in mind that $\xi$ is compactly supported in $\Omega_T$, to obtain
	\begin{align*}
	-\frac{1}{2} \int_{0}^T\!\! \int_{\Omega}  \xi \xi_t V^\epsilon \, dx dt &= \frac{2-p}{2}\int_{0}^T\!\! \int_{\Omega}
	\xi^2 (V^\epsilon)^{-\frac{p}{2}} \left \langle \nabla u^\epsilon, \nabla v^\epsilon \right \rangle \text{div}\Big( (V^\epsilon)^{\frac{p-2}{2}} \nabla u^\epsilon \Big) \, dx dt \\[1em]
	& -  \int_{0}^T \!\! \int_{\Omega} \frac{\partial }{\partial x_j} \left \{  (V^\epsilon)^{\frac{p-2}{2}} u^\epsilon_k \right \} \, \frac{\partial }{\partial x_k} \left \{    \xi^2 (V^\epsilon)^{\frac{2-p}{2}}  u^\epsilon_j \right \}  dx dt.
	\end{align*}
	\\
	Writing out the derivatives gives the fundamental formula
        
	\bigskip
        
	\begin{align*}
	&\int_{0}^T  \!\! \int_{\Omega} \xi^2 |D^2 u^\epsilon|^2 \, dx dt \qquad \qquad \qquad \text{\sf Main Term} \quad &(I) \\[1em]
	&\quad{}+ \frac{p-2}{2} \int_{0}^T\!\! \int_{\Omega}  \frac{1}{V^\epsilon}\, \xi^2 \left \langle \nabla u^\epsilon, \nabla v^\epsilon \right \rangle \Delta u^\epsilon \, dx dt \quad &(II) \\[1em]
	& = \,\frac{1}{2} \int_{0}^T\!\! \int_{\Omega}  \xi \xi_t V^\epsilon \, dx dt \quad &(III) \\[1em]
	&\quad{}+ (2-p) \int_{0}^T\!\! \int_{\Omega} \frac{1}{V^\epsilon}\xi  \left \langle \nabla u^\epsilon, \nabla v^\epsilon \right \rangle \left \langle \nabla u^\epsilon, \nabla \xi \right \rangle \, dx dt \quad &(IV) \\[1em]
	&\quad{}-  \int_{0}^T\!\! \int_{\Omega} \xi  \left \langle \nabla v^\epsilon, \nabla \xi \right \rangle \, dx dt \quad &(V).
	\end{align*}

        \bigskip
        
        In the next section we  shall bound the Main Term (I) uniformly with respect to $\epsilon.$ 
	
	\section{Estimate of the second derivatives}

	We shall  provide an estimate of the main term (I).  First, we record the elementary inequality \\
		\begin{equation}
		|\nabla v^\epsilon |^2=  \left|  2  D^2 u^\epsilon   \nabla u^\epsilon \right |^2 \leq 4 |D^2 u^\epsilon|^2 v^\epsilon.\\[1em]
		\label{nablaveps}
		\end{equation}
		\textbf{One Dimension.}  As an exercise, we show that in this case the second derivatives are locally bounded in $L^2$ for any $1<p<\infty$. In one dimension, equation \eqref{pLaplace} reads \\ $$u_t=|u_x|^{2-p}\frac{\partial }{\partial x} \left \{ |u_x|^{p-2}u_x \right \} = (p-1)u_{xx}.$$ 
		We absorb the terms (IV) and (V), using Young's inequality and inequality \eqref{nablaveps}.  For any $\delta>0$, \\
		\begin{align*}
		&  \int_0^T\!\!\int_{\Omega} \xi^2  \Big(\frac{\partial^2 u^\epsilon}{\partial x^2}\Big)^2 \left( 1 + (p-2)\,\dfrac{(\frac{\partial u^\epsilon}{\partial x})^2}{(\frac{\partial u^\epsilon}{\partial x})^2+\epsilon^2} - \delta \left( |p-2| + 1\right) \right)\, dx dt \\[1em]
		& \leq \frac{1}{2} \int_{0}^T\!\!\int_{\Omega}  \xi\, \xi_t V^\epsilon \, dx dt \,  + \,\frac{|p-2|+1 }{\delta}  \int_0^T\!\!\int_{\Omega} V^\epsilon |\nabla \xi|^2 \, dx dt. \\
		\end{align*}

		Applying Theorem \ref{Kerstingjor} we see that the right-hand side is bounded by a constant independent of $\epsilon>0.$  We have
                $$1 + (p-2)\,\dfrac{(\frac{\partial u^\epsilon}{\partial x})^2}{(\frac{\partial u^\epsilon}{\partial x})^2+\epsilon^2}\, \geq\, \min\{1,p-1\}\,  >0.$$ \\[1em]
                It follows that $\frac{\partial^2 u^\epsilon}{\partial x^2} \in L^2$ locally for any $p\in (1,\infty)$.  
		\\

	\textbf{General $n$.} We assume for the moment that $1<p<2$. 
	We rewrite the term (II) involving the Laplacian as \\
	\begin{align*}
	\tfrac{2-p}{2}  \frac{1}{V^\epsilon}\xi^2  \left \langle \nabla u^\epsilon, \nabla v^\epsilon \right \rangle \Delta u^\epsilon = \tfrac{2-p}{2}\frac{1}{V^\epsilon}\xi  \left \langle \nabla u^\epsilon, \nabla v^\epsilon \right \rangle \left \{  \Delta (\xi u^\epsilon) - 2 \left \langle \nabla u^\epsilon, \nabla \xi \right \rangle - u^\epsilon \Delta \xi   \right \}. \\
	\end{align*}
	Upon this rewriting the term (IV) disappears from the equation. We focus our attention on the term involving $\Delta(\xi u^\epsilon)$.  
        By  Lemma \ref{evanstalenti}
	\begin{align*}
	\int_0^T \!\!\int_{\Omega} |D^2 (\xi u^\epsilon)|^2 \,  dx  dt= \int_0^T\!\! \int_{\Omega} |\Delta (\xi u^\epsilon)|^2 \, dx dt. \\
	\end{align*}
	Differentiating, we see that 
	\begin{align*}
	&(\xi u^\epsilon )_i = \xi_i u^\epsilon + \xi u^\epsilon_i \\
	&(\xi u^\epsilon )_{ij} = \xi_{ij}u^\epsilon + u^\epsilon_i \xi_j +  \xi_i u^\epsilon_j + \xi u^\epsilon_{ij}. 
	\end{align*}
	It follows that\\
	\begin{align*}
	|D^2 (\xi u^\epsilon )|^2 = \xi^2 |D^2 u^\epsilon|^2 + f(u^\epsilon, \nabla u^\epsilon,D^2 u^\epsilon),\\
	\end{align*}
	where $f(u^\epsilon, \nabla u^\epsilon_i, D^2u^\epsilon)$\footnote{\begin{align*}
		f(u^\epsilon, \nabla u^\epsilon, D^2u^\epsilon)&= (u^\epsilon)^2|D^2 \xi|^2 +4u^\epsilon \left \langle \nabla \xi, D^2 \xi \nabla u^\epsilon \right \rangle + 4 \xi \left \langle \nabla \xi, D^2 u^\epsilon \nabla u^\epsilon \right \rangle   \\
		& +  2 |\nabla \xi|^2 |\nabla u^\epsilon|^2 + 2\left |\left \langle \nabla u^\epsilon, \nabla \xi \right \rangle \right |^2 + 2 u^\epsilon \xi\, \text{trace}\left \{ \left(D^2 \xi \right)\left(D^2 u^\epsilon \right)\right \}.
		\end{align*}} depends only linearly  on the second derivatives $u^\epsilon_{ij}$. By Young's inequality we obtain \\
	\begin{align*}
 \frac{2-p}{2} 	\int_0^T\!\! \int_{\Omega}  \frac{1}{V^\epsilon}\xi  \left \langle \nabla u^\epsilon, \nabla v^\epsilon \right \rangle \Delta (\xi u^\epsilon)\,dx dt&\, \leq\, \frac{5}{4}(2-p)  \int_0^T\!\! \int_{\Omega} \xi^2 |D^2 u^\epsilon|^2 \, dx dt  \\
	&\quad{}+\,\frac{2-p}{4} \int_0^T\!\! \int_{\Omega}  f(u^\epsilon, \nabla u^\epsilon,D^2 u^\epsilon) \, dx dt. 
	\end{align*}
	
	Inserting this into the main equation gives \\
	\begin{align*}
	&\Big( 1 - \frac{5}{4}(2-p) \Big)  \int_{0}^T \!\!\int_{\Omega}  \xi^2  |D^2 u^\epsilon|^2 \, dx dt \quad &(I*) \\[1em]
	& \leq  \frac{1}{2} \int_{0}^T \!\!\int_{\Omega}  \xi \xi_t V^\epsilon \, dx dt \quad &(III) \\[1em]
	&\quad{}-  \int_{0}^T \!\!\int_{\Omega}  \xi \left \langle \nabla v^\epsilon, \nabla \xi \right \rangle \, dx dt \quad &(V) \\[1em]
	& \quad{}+\frac{2-p}{2}  \int_0^T \!\!\int_{\Omega}  f(u^\epsilon, u^\epsilon_i, u^\epsilon_{ij}) \, dx dt \quad &(VI) \\[1em]
	& \quad{}+ \frac{2-p}{2}  \int_0^T \!\!\int_{\Omega}  \frac{1}{V^\epsilon}\xi  \left \langle \nabla u^\epsilon, \nabla v^\epsilon \right \rangle u^\epsilon \Delta \xi \, dx dt \quad &(VII). \\[1em]
	\end{align*}
	
All terms containing $D^2u^\epsilon$ can be absorbed by the new main term ($I*$). To this end, we use Young's inequality with a small parameter $\delta>0$ to balance\footnote{The parameter $\delta$ is to be made so small that terms like $\delta  \int_0^T \!\!\int_{\Omega}  \xi^2 |D^2 u^\epsilon|^2 \, dx dt $ can be moved over to the left-hand side.} the terms.  For term (V), we have \\
	\begin{align*}
	& \int_{0}^T \!\!\int_{\Omega}  \xi \left \langle \nabla v^\epsilon, \nabla \xi \right \rangle \, dx dt \leq \delta  \int_0^T \!\!\int_{\Omega}  \xi^2 |D^2 u^\epsilon|^2 \, dx dt + \frac{1}{\delta} \int_0^T \!\!\int_{\Omega}  V^\epsilon |\nabla \xi|^2 \, dx dt. \\
	\end{align*}
	Similarly, for term (VII) 
	
	\begin{align*}
	 \int_0^T \!\!\int_{\Omega} \frac{1}{V^\epsilon}\xi  \left \langle \nabla u^{\epsilon}, \nabla  v^\epsilon \right \rangle u^{\epsilon} \Delta \xi \, dx dt  \leq 2\delta_1  \int_0^T \!\!\int_{\Omega} \xi^2 |D^2 u^\epsilon|^2  + \frac{1}{\delta_1} \int_0^T \!\!\int_{\Omega} |u^\epsilon|^2 |\Delta \xi|^2 \, dx dt.  
	\end{align*}
	Using similar inequalities for the term involving $f(u^\epsilon, \nabla u^\epsilon,D^2 u^\epsilon)$  and chosing the  parameters small enough in Young's inequality, we find,
	\begin{equation} 
	 \int_{0}^T\!\!\int_{\Omega} \xi^2  |D^2 u^\epsilon|^2 \, dx dt \, \leq \, C \underset{\{\xi\neq 0\}}{ \int\!\!\int}((u^\epsilon)^2+|\nabla u^\epsilon|^2)\,dx dt \\[1em]
	\label{mainreseq}
	\end{equation}
	where $C$ is independent of $\epsilon$ but depends on $\|\xi\|_{C^2}$, provided that
	\begin{align*}
	1-\frac{5}{4}(2-p)\, >\, 0, \quad \text{i.e.} \quad
	p\,> \,\dfrac{6}{5}.
        \end{align*}
     This is now a decisive restriction.   Invoking Lemma \ref{lemma1} and the estimate (\ref{gradbound}), we deduce that that the majorant in (\ref{mainreseq}) is independent of $\epsilon.$
	
	A symmetric proof when $p>2$ shows that equation \eqref{mainreseq} holds when \\
	\begin{align*}
	p< \frac{14}{5}. \\
	\end{align*}

\section{The case $1<p<2$}

In this section, we give a proof of Theorem \ref{teor2}. To this end, let $\xi \in C_0^{\infty}(\Omega_T)$, with $0\leq \xi \leq 1$. We claim that

\begin{equation}
\int_0^T \!\!\int_{\Omega} \xi^2 \Bigl(\frac{\partial u^\epsilon}{\partial t}\Bigl)^2  \, dx dt  \leq 4 ||V^\epsilon||^2_\infty \left\{  \int_0^T \!\!\int_{\Omega} \left | \nabla \xi \right |^2 \, dx dt +\frac{1}{p} \int_0^T \!\!\int_{\Omega} \xi |\xi_t| \, dx dt \right \}
\label{helptime}
\end{equation}	
where the supremum norm of $V^\epsilon= \left| \nabla u^\epsilon \right |^2 + \epsilon^2$ is taken locally, over the support of $\xi$. Here, $u^\epsilon$ is the solution of the regularized equation \eqref{pLaplaceepsilon}. This is enough to complete the proof of Theorem \ref{teor2}, in virtue of Theorem \ref{Kerstingjor}.\\

Multiplying the regularized equation \eqref{pLaplaceepsilon} by $\left( |\nabla u^\epsilon|^2+ \epsilon^2 \right)^{\frac{p-2}{2}}\xi^2 u_t^\epsilon$ yields
\begin{align*}
&\xi^2 \left( |\nabla u^{\epsilon}|^2 + \epsilon^2 \right)^{\frac{p-2}{2}}\left( u_t^\epsilon \right)^2=\xi^2 u_t^\epsilon \text{div}\left( (|\nabla u^\epsilon|^2 + \epsilon^2 )^{\frac{p-2}{2}} \nabla u^\epsilon \right) \\[1em]
&= \text{div}\left(\xi^2 u_t^\epsilon (|\nabla u^\epsilon|^2 + \epsilon^2 )^{\frac{p-2}{2}} \nabla u^\epsilon \right) - (|\nabla u^\epsilon|^2 + \epsilon^2 )^{\frac{p-2}{2}} \left \langle \nabla u^\epsilon, \nabla \left(\xi^2 u_t^\epsilon  \right) \right \rangle.
\end{align*}
The integral of the divergence term vanishes by Gauss's Theorem and, upon integration, we have
 
\begin{align*}
&\int_0^T\!\! \int_{\Omega} \xi^2 \left(V^\epsilon\right)^{\frac{p-2}{2}}\left( u_t^\epsilon \right)^2  \,dx dt \\[1em]
&=  - \int_0^T\!\! \int_{\Omega} \left(V^\epsilon\right)^{\frac{p-2}{2}} \left \langle \nabla u^\epsilon, \nabla \left(\xi^2 u_t^\epsilon  \right) \right \rangle \, dx dt  \\[1em]
&= -2 \int_0^T\!\! \int_{\Omega} \xi \left(V^\epsilon\right)^{\frac{p-2}{2}}\left \langle \nabla u^\epsilon, \nabla \xi \right \rangle u_t^\epsilon\, dx dt - \int_0^T\!\! \int_{\Omega} \xi^2 \left(V^\epsilon\right)^{\frac{p-2}{2}}\left \langle \nabla u^\epsilon, \nabla u_t^\epsilon \right \rangle\, dx dt.
\end{align*}
The first integral on the right-hand side can be absorbed by the left-hand side by choosing $\sigma=\frac{1}{2}$ in 
\begin{align*}
\left | 2\xi \left(V^\epsilon\right)^{\frac{p-2}{2}} \left \langle \nabla u^\epsilon, \nabla \xi \right \rangle u_t^\epsilon  \right | \leq \sigma \xi^2 \left(V^\epsilon\right)^{\frac{p-2}{2}} \left( u_t^\epsilon \right)^2 + \frac{1}{\sigma}\left(V^\epsilon\right)^{\frac{p-2}{2}} |\nabla u^\epsilon|^2 |\nabla \xi|^2,
\end{align*} 
and integrating. \\

For the last term, the decisive observation is that 
\begin{align*}
\frac{1}{p}\frac{\partial }{\partial t} \left( |\nabla u^\epsilon|^2 + \epsilon^2 \right)^{\frac{p}{2}}= \left( |\nabla u^\epsilon|^2 + \epsilon^2 \right)^{\frac{p-2}{2}} \left \langle \nabla u^\epsilon, \nabla u_t^\epsilon \right \rangle= \left(V^\epsilon\right)^{\frac{p-2}{2}} \left \langle \nabla u^\epsilon, \nabla u_t^\epsilon \right \rangle.
\end{align*}
We use this in the last integral on the right-hand side to obtain

\begin{align*}
& - \int_0^T\!\! \int_{\Omega} \xi^2 \left(V^\epsilon\right)^{\frac{p-2}{2}}\left \langle \nabla u^\epsilon, \nabla u_t^\epsilon \right \rangle\, dx dt \\[1em]
&= - \int_0^T\!\! \int_{\Omega} \frac{\partial }{\partial t}\left \{ \frac{\xi^2}{p} \left ( V^\epsilon \right)^\frac{p}{2} \right \} \, dx dt  +  \frac{2}{p}\int_0^T\!\! \int_{\Omega} \xi \xi_t \left( V^\epsilon \right)^\frac{p}{2} \, dx dt  \\[1em]
&=  - \int_{\Omega} \left [ \frac{\xi^2}{p}\left ( V^\epsilon \right)^\frac{p}{2}  \right]_{t=0}^{t=T} \, dx  +  \frac{2}{p}\int_0^T\!\! \int_{\Omega} \xi \xi_t \left( V^\epsilon \right)^\frac{p}{2} \, dx dt  \\[1em]
&=   \frac{2}{p}\int_0^T\!\! \int_{\Omega} \xi \xi_t \left( V^\epsilon \right)^\frac{p}{2} \, dx dt. 
\end{align*}
To sum up, we have now the final estimate 

\begin{align*}
&\frac{1}{2} \int_0^T\!\! \int_{\Omega} \xi^2 \left(V^\epsilon\right)^{\frac{p-2}{2}}\left( u_t^\epsilon \right)^2  \,dx dt \\[1em]
&\leq 2  \int_0^T\!\! \int_{\Omega} \left(V^\epsilon\right)^{\frac{p-2}{2}} |\nabla u^\epsilon|^2 |\nabla \xi|^2 \, dx dt + \frac{2}{p}\int_0^T\!\! \int_{\Omega} \xi \xi_t \left( V^\epsilon \right)^\frac{p}{2} \, dx dt \\[1em]
& \leq  2  \int_0^T\!\! \int_{\Omega} \left(V^\epsilon\right)^{\frac{p}{2}} |\nabla \xi|^2 \, dx dt + \frac{2}{p}\int_0^T\!\! \int_{\Omega} \xi \xi_t \left( V^\epsilon \right)^\frac{p}{2} \, dx dt.
\end{align*}
So far, our calculations are valid in the full range $1<p<\infty$. For $1<p<2$, we have $$\left( V^\epsilon \right)^{\frac{p-2}{2}} \geq ||V^\epsilon||_\infty^{\frac{p-2}{2}},$$ where the supremum norm is taken over the support of $\xi$. Hence, equation \eqref{helptime} holds for $1<p<2$ and the proof of Theorem \ref{teor2} is complete.


\begin{thebibliography}{ABCD}{\small
			
			\bibitem[APR]. \textsc{A. Attouchi, M. Parviainen and E. Ruosteenoja}: {$C^{1,\alpha}$ regularity for the normalized p-Poisson problem}, \textit{Journal de Math{\'e}matiques Pures et Appliqu{\'e}es}, 108(4):553-591, 2017.
			\bibitem[BG]. \textsc{A. Banerjee and N. Garofalo}: {Gradient bounds and monotonicity of the energy for some nonlinear singular diffusion equations}, \textit{Indiana University Mathematics Journal}, pages 699-736, 2013.		
			\bibitem[D]. \textsc{K. Does}: {An evolution equation involving the normalized $ p $-Laplacian}, \textit{Communications on Pure \& Applied Analysis}, 10(1):361-396, 2011.	
			\bibitem[E]. \textsc{L. Evans}: {\it Partial Differential Equations}, volume 19 of \textit{Graduate students in Mathematics}. Provindence, Rhode Island, 1998. 
			\bibitem [ES]. \textsc{L. Evans and J. Spruck}: {Motion of level sets by mean curvature}. \textit{Journal of Differential Geometry}, 33(3):635-681, 1991. 
			\bibitem[JS] .\textsc{T. Jin and L. Silvestre}: {H\"{o}lder gradient estimates for parabolic homogeneous p-Laplacian equations}. \textit{Journal de Math{\'e}matiques Pures et Appliqu{\'e}es}, 108(1):63-87, 2017.
			\bibitem[MPR]. \textsc{J. Manfredi, M. Parviainen and J. Rossi}: {An asymptotic mean value characterization for a class of nonlinear parabolic equations related to tug-of-war games.} \textit{SIAM Journal on Mathematical Analysis}, 42(5):2058-2081, 2010.
			
			
		}
		
		
		
	\end{thebibliography}
\end{document}